\documentclass[11pt,a4paper]{article}
\usepackage{amssymb}
\usepackage{amsmath}
\usepackage{float}
\usepackage{subcaption}
\usepackage{tikz}

\font\Bbb=msbm10 at 10 truept
\def\n{\hbox{\Bbb N}}

\DeclareMathOperator{\Reg}{Reg}
\DeclareMathOperator{\Base}{Base}

\newcommand{\qed}{
  \ifmmode
   \eqno{\qedsymbol}
  \else
    \leavevmode\unskip\penalty9999 \hbox{}\nobreak\hfill\hbox{\qedsymbol}
  \fi
}
\newcommand{\qedsymbol}{\leavevmode\vrule height 1.2ex width 1.1ex depth -.1ex}
\newenvironment{proof}{\begin{trivlist}\item[\hskip
\labelsep{\bf Proof.\quad}]}
{\hfill\qed\rm\end{trivlist}}

\newtheorem{theorem}{Theorem}[section]
\newtheorem{corollary}[theorem]{Corollary}
\newtheorem{proposition}[theorem]{Proposition}
\newtheorem{lemma}[theorem]{Lemma}

\title{Semilattices of Stratified Extensions}
\author{James Renshaw\\
\small School of Mathematical Sciences\\
\small University of Southampton\\
\small Southampton, SO17 1BJ, England\\
\small ORCID: 0000-0002-5571-8007\\
\small  j.h.renshaw@soton.ac.uk\\\\
William Warhurst\\
\small School of Mathematical Sciences\\
\small University of Southampton\\
\small Southampton, SO17 1BJ, England\\
\small  w.warhurst@soton.ac.uk\\
}

\begin{document}
\date{May 2023}
\maketitle
\begin{abstract}
\noindent  In 1995 Grillet introduced the concept of a stratified semigroup as a kind of generalisation of finite nilsemigroups. We extend these ideas here by allowing a more general $\Base$ and describe them in terms of extensions of semigroups by stratified semigroups. We consider semillatices of certain types of group-bound semigroups and also semillatices of Clifford semigroups and show how to describe them as semilattices of these stratified extensions and provide a number of interesting examples.
\\
\\
{\bf Keywords} Semigroup, stratified, extension, semilattice, group-bound, Clifford semigroup.\\
{\bf Mathematics Subject Classification} 2020: 20M10.
\end{abstract}

\section{Introduction and Preliminaries}

\smallskip

Grillet \cite{grillet-95} defines a semigroup $S$ with zero to be \textit{stratified} whenever $\bigcap_{m>0}S^m = \{0\}$.
A semigroup without zero is called \textit{stratified} if $S^0$ is stratified. He shows that this class of semigroups includes the class of all free semigroups, free commutative semigroups, homogeneous semigroups and nilpotent semigroups with finite index. Our aim is to generalise this concept and consider some semigroups that can be decomposed as semilattices of some of these more general kinds of stratified semigroups.

\medskip

After some basic definitions and preliminary results, in section 2, we introduce the concept of a {\em stratified extension} as a generalisation of Grillet's stratified semigroups, and we provide a number of interesting results on the overall structure of such semigroups. The final two sections then examine two families of semigroups that exhibit this stratified structure. In section 3 our focus is on semigroups in which every regular ${\cal H}-$class contains an idempotent. We show that group-bound semigroups with this property are semilattices of stratified extensions of completely simple semigroups and describe the semilattice structure. Finally in section 4 we look at strict extensions of Clifford semigroups and show amongst other things that strict stratified extensions of Clifford semigroups are semilattices of stratified extensions of groups. For all terminology in semigroups not otherwise defined see Howie (\cite{howie-95}).

\bigskip

Let $S$ and $T$ be semigroups, with $T$ containing a zero. A semigroup $\Sigma$ is called an \textit{ideal extension} of $S$ by $T$ if it contains $S$ as an ideal and the Rees quotient $\Sigma/S$ is isomorphic to $T$. Grillet and Petrich \cite{grillet-petrich-68} define an extension as \textit{strict} if every element of $\Sigma \setminus S$ has the same action on $S$ as some element of $S$ and \textit{pure} if no element of $\Sigma \setminus S$ does. They also showed that any extension of an arbitrary semigroup $S$ is a pure extension of a strict extension of $S$.

\begin{proposition}[{\cite[Proposition 2.4]{grillet-petrich-68}}]\label{grillet-petrich-proposition}
Every extension of $S$ is strict if and only if $S$ has an identity.
\end{proposition}

Let $S$ and $T$ be disjoint semigroups. A {\em partial homomorphism}~\cite{clifford-preston-61} from $T$ to $S$ is a map $f:T\setminus\{0\}\to S$ such that for all $x,y \in S, f(xy) = f(x)f(y)$ whenever $xy\ne0$.

We adopt the convention used by Clifford and Preston (\cite{clifford-preston-61}) that elements of $T \setminus \{0\}$ are denoted by capital letters and elements of $S$ by lowercase letters. A partial homomorphism from $T \setminus \{0\}$ to $S$ given by $A\mapsto \overline{A}$ defines an extension $\Sigma = S\bigcup T\setminus\{0\}$ with multiplication given by
\begin{enumerate}
\item $A\ast B =
\begin{cases}
AB & AB \neq 0\\
\overline{A}\;\overline{B} & AB = 0
\end{cases}$
\item $A\ast s = \overline{A}s$
\item $s\ast A = s\overline{A}$
\item $s\ast t = st$
\end{enumerate}
where $A,B \in T \setminus \{0\}$ and $s,t \in S$. From parts (2) and (3) above, all extensions defined in this way are strict.

\smallskip

Let $S$ be a semigroup and let $a,b\in S$. We say that $a$ and $b$ are {\em interchangeable} if
$$
\forall x\in S, ax=bx\text{ and } xa=xb.
$$
A semigroup is called {\em weakly reductive} if it contains no interchangeable elements. Notice that every monoid is weakly reductive.

\begin{theorem}[{\cite[Theorem 2.5]{grillet-petrich-68}}]\label{grillet-petrich-theorem}
Let $S$ be weakly reductive. Then every strict extension of $S$ is determined by a partial homomorphism, and conversely.
\end{theorem}

Recall that a semigroup is said to be {\em $E-$dense} (or {\em $E-$inversive}) if for all $s \in S$ there exists $t \in S$ such that $st\in E(S)$. The following is well-known

\begin{lemma}
The following are equivalent
\begin{enumerate}
\item $S$ is $E-$dense,
\item for all $s \in S$ there exists $t\in S$ such that $ts\in E(S)$,
\item for all $s \in S$ there exists $t\in S$ such that $st,ts\in E(S)$,
\item for all $s\in S$ there exists $s'\in S$ such that $s'ss' = s'$.
\end{enumerate}
\end{lemma}

Such an element, $s'$, in (4) is called a {\em weak inverse} of $s$ and the set of all weak inverses of $s$ is denoted by $W(s)$. The set of all weak inverses of elements of $S$ is denoted by $W(S)$. Note that $W(S)=\Reg(S)$. It is easily shown that for all $s\in S, s'\in W(s)$, $ss', s's \in E(S)$ and $ss' \mathcal{L} s' \mathcal{R} s's$.

\begin{lemma} \label{E-dense-J-lemma}
Let $S$ be a semigroup.
\begin{enumerate}
\item If $s,t \in S$ then $W(st) \subseteq W(t)W(s)$.
\item If $S$ is an $E-$dense semigroup and $s' \in W(s)$ then $J_{s'} \leq J_s$.
\end{enumerate}
\end{lemma}

\begin{proof}
These are fairly straightforward.
\begin{enumerate}
\item Let $(st)' \in W(st)$. Then $(st)'=(st)'st(st)'$ and so $t(st)'=t(st)'st(st)'$ and hence $t(st)' \in W(s)$. Similarly $(st)'s \in W(t)$. Then $(st)' = ((st)'s)(t(st)') \in W(t)W(s)$.
\item Since $s' \in W(s)$, $s'=s'ss'$ and so $J_{s'} = J_{s'ss'}$. By \cite[Equation 2.1.4]{howie-95} we have $J_{s'}=J_{s'ss'} \leq J_s$.
\end{enumerate}
\end{proof}

Note that it is well known that if $E(S)$ forms a band then $W(st)=W(t)W(s)$.

\medskip

An element $s$ in a semigroup $S$ is called {\em eventually regular} if there exists $n\ge 1$ such that $s^n$ is regular. A semigroup is {\em eventually regular} if all of its elements are eventually regular. It is clear that eventually regular semigroups are $E-$dense. A semigroup $S$ is called {\em group-bound} if for every $s \in S$, there exists $n\ge1$ such that $s^n$ lies in a subgroup of $S$. Clearly group-bound semigroups are eventually regular. If $S$ is eventually regular and each regular ${\cal H}-$class is a group then $S$ is group-bound.

\smallskip

A semigroup $S$ is called \textit{Archimedean} if for any $a, b \in S$ there exists $n \in \n$ such that $a^n \in SbS$.

\begin{theorem}[{\cite[Theorem 3]{shevrin-95}}]\label{shevrin-theorem-3}
Let $S$ be a group-bound semigroup. Then $S$ is a semillatice of Archimedean semigroups if and only if every regular ${\cal H}-$class of $S$ is a group.
\end{theorem}

Let $S$ be a semigroup with $0$. We say that an element $x\in S$, is {\em nilpotent} if there is $n\in \n$ such that $x^n=0$. The semigroup $S$ is called a nilsemigroup if every element of $S$ is nilpotent. The semigroup $S$ is called {\em nilpotent with index} $n\in\n$ if $S^n=\{0\}$.

\section{Stratified Extensions}

Let $S$ be a semigroup (not necessarily stratified) and define the {\em base} of $S$ to be the subset $\Base(S) = \bigcap_{m>0} S^m$. We shall say that a semigroup $S$ is a \textit{stratified extension} of $\Base(S)$ if $\Base(S)\ne\emptyset$. The reason for this name will become apparent later. Clearly $\Base(S)$ is a subsemigroup of $S$. When $\Base(S)$ is a trivial subgroup then $S$ is a stratified semigroup. A stratified semigroup $S$ is not in general a stratified extension as we may have $\Base(S) = \emptyset$, however if $S$ is a stratified semigroup then $S^0$ is also stratified and is a stratified extension with trivial base. Further, $S$ is called a \textit{finitely stratified extension} if there exists $m \in \mathbb{N}$ such that $S^m = S^{m+1} = \Base(S)$. The smallest such $m$ is called the \textit{height} of $S$ and where necessary we shall refer to $S$ as a {\em finitely stratified extension with height $m$}. If for every $s$ in $S$ there is an $m \in \n$ such that $s^m \in \Base(S)$ then $S$ is a \textit{nil-stratified extension}. All finitely stratified extensions are nil-stratified extensions, but it is easy to demonstrate that not all nil-stratified extensions are finitely stratified extensions.

A finitely stratified extension is a stratified extension over the same base, since $S^m = S^{m+1}$ implies $S^n = S^m$ for all $n \geq m$ and so $\bigcap_{k>0} S^k = S^m$, where $m$ is the height of $S$. The converse is not true since, for example, if $S$ is a free semigroup with a zero adjoined, then $S$ is a stratified extension with trivial base but not a finitely stratified extension. It is clear that a (finitely) stratified extension has a unique base.

\medskip

Clearly, for all $m\ge1$, $S^{m+1} \subseteq S^{m}$ and so we define the \textit{layers} of $S$ as the sets $S_m = S^m \setminus S^{m+1}$, $m\ge 1$. Every element of $S \setminus \Base(S)$ lies in exactly one layer, and if $s \in S_m$ then $m$ is the \textit{depth} of $s$. The layer $S_1$ generates every element of $S \setminus \Base(S)$ and is contained in any generating set of $S$. However $\Base(S)\not\subseteq \langle S_1\rangle$ in general. For example, let $S$ be a semigroup with $0$, with no zero divisors. Then $0\in \Base(S)$ but $0 \not\in \langle S_1 \rangle$.

\medskip

Since $\Base(S) \subseteq S^m$ for any $m \in \mathbb{N}$, we have an alternative characterisation for the elements of $\Base(S)$.
Any $s \in S$ lies in $\Base(S)$ if and only if $s$ can be factored into a product of $m$ elements for any $m \in \mathbb{N}$, i.e. $s = a_1 a_2 \dots a_m$ for some $a_i \in S$. This characterisation gives us some immediate properties of $\Base(S)$ as a subsemigroup of $S$.

\begin{lemma}\label{lemma-2-1}
Let $S$ be a semigroup and let $s\in S$. If $s\in Ss\cup sS\cup SsS$ then $s \in \Base(S)$.
\end{lemma}
\begin{proof}
It follows that for $m\ge1$, $s=x^{m}s$ or $s=sy^{m}$ or $s=x^{m}sy^{m}$ and so the result follows from the previous observation.
\end{proof}

\begin{corollary}
Suppose that $S$ is a semigroup.
\begin{enumerate}
\item Every monoid subsemigroup of $S$ is a submonoid of $\Base(S)$.
\item $\Reg(S)\subseteq \Base(S)$. Hence if $S$ is regular, $\Base(S)=S$.
\item $E(S)=E(\Base(S))$.
\item If $s \in S \setminus \Base(S)$ then $|J_s| = 1$, where $J_s$ is the ${\mathcal J}-$class of $s$.
\end{enumerate}
\end{corollary}

To see (4) notice that if $a \mathcal{J} b$ and $a \neq b$ then we have $a = ubv$ for some $u,v \in S^1$ and since $a \neq b$ we have $u$ and $v$ not both equal to $1$. Similarly $b = sat$ with $s,t\in S^1$ not both equal to $1$ and hence $a\in Sa\cup aS\cup SaS$.
The converse is not true, since for example in a semigroup with zero we have $J_0 = \{0\}$ but $0 \in \Base(S)$.

\smallskip

If follows immediately that the class of stratified extensions contains the class of semigroups with regular elements and hence in particular the classes of monoids, finite semigroups and regular semigroups. However, not every semigroup is a stratified extension. Consider for example a semigroup with a length function (i.e. a function $l:S\to\n$ such that for all $x,y \in S, l(xy) = l(x)+l(y)$). If $T$ is the subsemigroup of elements with non-zero length, then the elements of $T^m$ each have length at least $m$. Hence the elements of length exactly $m$ lie in $T_m \not\subseteq \Base(T)$ and so the base is empty. In particular, a free semigroup is not a stratified extension, nor is the semigroup of polynomials of degree $\ge 1$ over any ring, under multiplication.

\begin{proposition}\label{stratified-names-lemma}
Let $S$ be a stratified extension. Then
\begin{enumerate}
\item $\Base(S)$ is an ideal of $S$ and $S$ is an ideal extension of $\Base(S)$ by a stratified semigroup with $0$.
\item If $S$ is a nil-stratified extension then it is an ideal extension of $\Base(S)$ by a nilsemigroup.
\item If $S$ is a finitely stratified extension then it is an ideal extension of $\Base(S)$ by a nilpotent semigroup of finite index.
\end{enumerate}
\end{proposition}

\begin{proof}
\begin{enumerate}
\item For any $u, v \in S^1$, $t \in \Base(S)$ and $m > 3$, we have $t \in S^{m-2}$ so $utv \in S^m$ and hence $utv \in \Base(S)$. Hence we can regard $S$ as being an ideal extension of $\Base(S)$ by $S/\Base(S)$ and note that $S/\Base(S)$ is a stratified semigroup with $0$.

\item If $S$ is a nil-stratified extension then it follows that for every $s \in S$ there exists $m\in\n$ such that $s^m \in \Base(S)$. Hence in the Rees quotient $S/\Base(S)$, $s^m = 0$ and so $S/\Base(S)$ is a nilsemigroup and $S$ is an ideal extension by a stratified nilsemigroup.

\item Recall that the nilpotency index of a semigroup is the smallest value $m$ such that every product of $m$ elements is zero. It is easy to see that if the nilpotency index of $S/\Base(S)$ is $m$ then the height of $S$ is $m$ and so $S$ is a finitely stratified extension. Conversely, any nilpotent semigroup $S$ of finite nilpotency index $m$ is a stratified semigroup with $S^m = \{0\}$. We have hence proved the following.
\end{enumerate}
\end{proof}

The converses of these results do not hold. To see this, let $S$ be a free semigroup and $T$ be the two element nilsemigroup. Then $T$ is a stratified semigroup with $0$ but an extension of $S$ by $T$ is not a stratified extension. Further, $T$ is a nilpotent semigroup of finite index and an extension of $S^0$ by $T$ is a stratified extension, but is not a finitely stratified nor nil-stratified extension.

\begin{proposition}
Let $S$ be a stratified extension.
\begin{enumerate}
\item If $S$ is a nil-stratified extension then $\Base(S)$ is periodic if and only if $S$ is periodic;
\item If $S$ is a nil-stratified extension then $\Base(S)$ is eventually regular if and only if $S$ is eventually regular;
\item $\Base(S)$ is $E-$dense if and only if $S$ is $E-$dense.
\end{enumerate}
 So a stratified extension with a periodic base is $E-$dense.
\end{proposition}
\begin{proof}
If $S$ is either periodic or eventually regular, then so is any subsemigroup and in particular $\Base(S)$. Conversely, if $\Base(S)$ is periodic (resp. eventually regular) and $s\in S$, then since $S$ is nil-stratified it follows that there exists $m\in\n$ such that $s^m\in\Base(S)$ and so there exists $n\in\n$ such that $s^{mn}$ is idempotent (resp. regular) and so $\Base(S)$ is periodic (resp. eventually regular).

For the third, let $\Base(S)$ be $E-$dense and let $s\in S$. Then for any $t\in \Base(S), ts\in \Base(S)$ and so there exists $u\in \Base(S)$ such that $uts\in E(\Base(S)) =E(S)$ and so $S$ is $E-$dense. Conversely suppose that $S$ is $E-$dense. Since $W(S) = \Reg(S)\subseteq \Base(S)$, then $\Base(S)$ is $E-$dense.
\end{proof}

Notice that periodic $\Rightarrow$ eventually regular $\Rightarrow$ $E-$dense $\Rightarrow$ stratified extension.

\medskip

There is in general little control over the base as a stratified extension can be constructed with any given semigroup as its base.

\begin{proposition}
Let $T$ and $R$ be any semigroups.Then there exists a stratified extension $S$ such that $T\subseteq\Base(S)$ and $S/T\cong R$. Moreover, if $R$ is stratified without a zero then $T=\Base(S)$.
\end{proposition}
\begin{proof}
Let $S = R \dot\cup T$ and define a binary operation $\ast$ on $S$  by $r_1\ast r_2 = r_1r_2$ for $r_1,r_2 \in R$, $t_1\ast t_2 = t_1t_2$ for $t_1,t_2 \in T$, and $r\ast t=t\ast r=t$ for $r \in R$ and $t \in T$. It is easy to verify that this operation is associative and so $(S, \ast)$ is a semigroup. Then $T \subseteq \bigcap_{m>0} (R \dot\cup T)^m$ and so $S$ is a stratified extension. Moreover, if we choose $R$ such that $\bigcap_{m>0} R^m=\emptyset$, for example $R=A^+$, a free semigroup, we see that $\bigcap_{m>0} (R \dot\cup T)^m = T$ and so we can obtain a stratified extension with base $T$.
\end{proof}

In contrast, the possible bases for a finitely stratified extension are much more restricted. Let $S$ be a finitely stratified extension with $T=\Base(S)$ and consider $T^2$. There exists $m \in \mathbb{N}$ such that $T = S^m$, so $T^2 = S^{2m}$. But by definition $S^m = S^{m+1} = S^{m+2} = \dots = S^{2m}$ and so $T^2 = T$. A semigroup $T$ satisfying $T^2=T$ is said to be {\em globally idempotent} and so the base of a finitely stratified extension is globally idempotent. Note also that if $S$ is globally idempotent, then $S$ is a finitely stratified extension in a trivial sense, with base $S$ and height 1.

\begin{proposition}
A semigroup $\Sigma$ is a finitely stratified extension if and only if it is an ideal extension of a globally idempotent semigroup by a nilpotent semigroup of finite index.
\end{proposition}

\begin{proof}
We need only justify the converse. Let $\Sigma$ be an ideal extension of a globally idempotent semigroup $S$ by a nilpotent semigroup $T$ of finite index $m$. Then $\Sigma^m = S$ and $S=S^2$ so $\Sigma^m = \Sigma^{2m}$ and as each $\Sigma^i \subseteq \Sigma^{i+1}$ it follows that $\Sigma^m = \Sigma^{m+1}$. Hence $\Sigma$ is a finitely stratified extension with base $S$ and height $m$.
\end{proof}

This is still a very broad class of semigroup, including among its members every monoid and every regular semigroup. It should also be noted that a globally idempotent semigroup need not contain idempotents, the Baer-Levi semigroup being one such example.

\begin{proposition}
There exists a finitely stratified extension of height $h$, for any $h\in\n$.
\end{proposition}
\begin{proof}
Let $S =\langle a\rangle$ be the monogenic semigroup of index $h$ and period $r$ and let $G=\{a^h,\ldots, a^{h+r-1}\}$ be the kernel of $S$, a finite cyclic group of order $r$. For any $x=a^{n_1}\ldots a^{n_h}\in S^h$ it follows that $x=a^{n_1+\ldots+n_h}$ and $n_1+\ldots+n_h\ge h$ so that $S^h\subseteq G$. However, since $a^l\in S^l$ then $S^l\not\subseteq G$ for any $l<h$. Since $S^{h+r-1}\subseteq S^h$, $S^{h+1}\subseteq S^h$ and $G\subseteq S^{h+r-1}$  then $S^h=S^{h+1}=G$ and $S$ is a finitely stratified extension with base $G$ and height $h$.
\end{proof}

\smallskip

\begin{theorem}
Let $S$ be a (finitely, nil-) stratified extension and let $S_i$ for $i \in I$ be a family of stratified extensions.
\begin{enumerate}
\item If $\rho$ is a congruence on $S$ then $S/\rho$ is a (finitely, nil-) stratified extension with base $\Base(S)/\rho$;
\item the direct product $\prod_{i \in I} S_i$ is a stratified extension with base $\prod_{i \in I} \Base(S_i)$;
\item if $|I|<\infty$ and each $S_i$ is a (finitely, nil-) stratified extension then the direct product $\prod_{i \in I} S_i$ is a (finitely, nil-) stratified extension with base $\prod_{i \in I} \Base(S_i)$.
\end{enumerate}
\end{theorem}

\begin{proof}
\begin{enumerate}
\item Let $S$ be a semigroup and let $\rho$ be a congruence on $S$. It is easy to see that for any $m \in \n$ we have $(S/\rho)^m = S^m/\rho$. Hence if $S$ is a stratified extension then $S/\rho$ is also a stratified extension, with $\Base(S/\rho) = \Base(S)/\rho$. Further, if $S$ is a finitely stratified or nil-stratified extension then so is $S/\rho$.

\item Let $S_i$ be a family of semigroups. Then $(\prod_{i \in I} S_i)^m = \prod_{i \in I} {S_i}^m$. Hence if each $S_i$ is a stratified extension, the product $\prod_{i \in I} S_i$ is also a stratified extension with $\Base(\prod_{i \in I} S_i) = \prod_{i \in I} \Base(S_i)$.

\item If $I$ is a finite set and each $S_i$ is a nil-stratified extension then so is $\prod_{i \in I} S_i$. Similarly if each $S_i$ is a finitely stratified extension then so is $\prod_{i \in I} S_i$. To see that we cannot remove the condition $|I| < \infty$, let $I = \n$ and for each $i \in I$ let $S_i$ be a finitely stratified (and hence nil-stratified) extension of height $i$. Then $\prod_{i \in I} S_i$ is a stratified extension but is neither a finitely stratified extension nor a nil-stratified extension.
\end{enumerate}
\end{proof}

Subsemigroups of (finitely, nil-) stratified extensions are not necessarily (finitely, nil-) stratified extensions. For example the bicyclic semigroup is a finitely stratified extension (in fact globally idempotent) but contains $(\mathbb{N}, +)$ as a subsemigroup which is free and hence not even a stratified extension. The class of (finitely, nil-) stratified semigroups therefore does not form a variety.

\medskip

Having considered direct products of stratified extensions, and in anticipation of the families of semigroups in the final two sections, we observe

\begin{proposition}
Let $S=\bigcup_{\alpha \in Y} S_\alpha$ be a semilattice of stratified extensions. Then $S$ is a stratified extension.
\end{proposition}

\begin{proof}
This follows easily on observing that $\bigcup_{\alpha \in Y} \Base(S_\alpha) \subseteq \bigcap_{m > 0} S^m$.
\end{proof}

\smallskip

In the case of finitely stratified extensions, we can construct a semilattice of finitely stratified extensions which is not a finitely stratified extension. Let $Y = \mathbb{N} \cup \{0\}$ be a semilattice under the multiplication $ij = 0$ for all $i,j \in Y$ with $i \neq j$. For each $i \in \mathbb{N}$ let $S_i$ be a finitely stratified extension with height $i$ and let $S_0$ be globally idempotent. Let $S$ be the union of each $S_i$ as a semilattice of semigroups over $Y$. Then $S^m = S_0 \cup \bigcup_{i \in \mathbb{N}} {S_i}^m$. If $i > m$ then there are elements in ${S_i}^m$ which are not in ${S_i}^{m+1}$ and so $S^m \neq S^{m+1}$ for any $m \in \mathbb{N}$.

\smallskip

\section{Semilattices of group bound semigroups}

In this and the following section we explore two families of examples of semigroups that can be decomposed as semilattices of stratified extensions. In this section we introduce the definition of a conditionally completely regular semigroup and, after some preliminary results, show that if such a semigroup is also group-bound then it is a semilattice of stratified extensions.

\smallskip

A semigroup in which every regular $\cal H$-class contains an idempotent is called a {\em conditionally completely regular semigroup}. This is clearly equivalent to every regular ${\cal H}-$class being a group.

\smallskip

Let $S$ be a conditionally completely regular semigroup.  We define a relation $\rho$ on $S$ by $s\rho t$ if and only if for every $\cal D$-class $D$ of $S$ we have
$$
W(s) \cap D \neq \emptyset \iff W(t) \cap D \neq \emptyset.
$$
Clearly $\rho$ is an equivalence relation. We will show that $\rho$ is in fact a congruence, and moreover that $S/\rho$ is a semilattice.

\medskip

We begin by establishing some properties of such semigroups.

\begin{lemma} \label{D-class-lemma}
Let $S$ be a conditionally completely regular semigroup.
\begin{enumerate}
\item Every regular $\cal D$-class of $S$ is a completely simple subsemigroup of $S$.
\item Let $s\in S$ and $s' \in W(s)$. Every $\cal H$-class of $D_{s'}$ contains a weak inverse of $s$. 
\end{enumerate}
\end{lemma}

\begin{proof}
These are fairly straightforward.
\begin{enumerate}
\item Let $D$ be a regular ${\cal D}-$class and let $a,b \in D$ and let $e$ be the idempotent lying in $L_a \cap R_b$. Then $ab \mathcal{L} eb \mathcal{R} ee = e$. Hence $ab \in D$ and $D$ is a completely simple subsemigroup of $S$.
\item Let $D$ be the (regular) ${\cal D}-$class containing $s'$ and let $r$ be an idempotent such that $r \mathcal{R} ss'$.
Then $ss'r = r$ and $rss'=ss'$, and it follows that $s'r \in W(s)$ and $s'\mathcal{R}s'r\mathcal{L}r$.
\begin{center}
\begin{tikzpicture}
\draw (0,0) -- (4,0) -- (4,4) -- (0,4) -- (0,0);
\draw (1,0) -- (1,4);
\draw (2,0) -- (2,4);
\draw (3,0) -- (3,4);
\draw (0,1) -- (4,1);
\draw (0,2) -- (4,2);
\draw (0,3) -- (4,3);
\node at (1.5,1.5) {$s'$};
\node at (1.5,3.5) {$ss'$};
\node at (0.5,3.5) {$r$};
\node at (0.5,1.5) {$s'r$};
\node (a) at (0.5,3) {};
\node (b) at (2.5,3) {};
\node (c) at (0.5,-0.1) {};
\node (d) at (2.5,-0.1) {};
\end{tikzpicture}
\end{center}
Let $I=S/{\cal R}$ and $\Lambda = S/{\cal L}$ and as is normal denote the ${\cal R}-$classes as $R_i$ ($i\in I$), the ${\cal L}-$classes as $L_\lambda$ ($\lambda\in\Lambda$) and the ${\cal H}-$class $R_i\cap L_\lambda$ as $H_{i\lambda}$. Suppose $s'\in R_j$ and $ss'\in R_i$. For each $\lambda\in\Lambda$, let $r_{i\lambda}$ be the idempotent in $H_{i\lambda}$ so that we produce a weak inverse of $s$, $s'_{j\lambda}$, in $H_{j\lambda}$. Let $s'_{j\lambda}s\in L_\mu$, and for each $k\in I$ let $l_{k\mu}$ be the idempotent in $H_{k\mu}$ and note that, using a similar argument to above, $l_{k\mu}s'_{j\lambda} \in W(s)$ and $l_{k\mu}s'_{j\lambda}\in H_{k\lambda}$.
\begin{center}
\begin{tikzpicture}
\draw (0,0) -- (4,0) -- (4,4) -- (0,4) -- (0,0);
\draw (1,0) -- (1,4);
\draw (2,0) -- (2,4);
\draw (3,0) -- (3,4);
\draw (0,1) -- (4,1);
\draw (0,2) -- (4,2);
\draw (0,3) -- (4,3);
\node at (3.5,1.5) {$s'_{j\lambda}s$};
\node at (0.5,3.5) {$r_{i\lambda}$};
\node at (3.5,2.5) {$l_{k\mu}$};
\node at (0.5,1.5) {$s'_{j\lambda}$};
\node at (0.5,2.5) {$l_{k\mu}s'_{j\lambda}$};
\node (a) at (0.5,3) {};
\node (b) at (2.5,3) {};
\node (c) at (0.5,-0.1) {};
\node (d) at (2.5,-0.1) {};
\end{tikzpicture}
\end{center}
\end{enumerate}
\end{proof}

Notice that the converse of the first point is true as well. The second point allows us to give an equivalent definition of $\rho$:  $s \rho t$ if and only if for every $\cal H$-class $H$ of $S$ we have 
$$
W(s) \cap H \neq \emptyset \iff W(t) \cap H \neq \emptyset.
$$
The next result is key in what follows.

\begin{lemma} \label{rho-mult-lemma}
Let $S$ be a conditionally completely regular semigroup and let $s, t \in S$. For any $\cal D$-class $D$ of $S$ we have
$$
W(st) \cap D \neq \emptyset\text{ if and only if }W(s) \cap D \neq \emptyset\text{ and }W(t) \cap D \neq \emptyset.
$$
\end{lemma}

\begin{proof}
Let $s' \in W(s) \cap D$ and suppose $W(t) \cap D \neq \emptyset$. Then $s's$ is an idempotent lying in $D$. By Lemma \ref{D-class-lemma}(2) $t$ has a weak inverse in every $\cal H$-class of $D$, so let $t'$ be the weak inverse of $t$ lying in the $\cal H$-class of $s's$. Then $t's'stt's' = t'tt's' = t's'$. By Lemma \ref{D-class-lemma}(1) $t's' \in D$ and so $W(st) \cap D \neq \emptyset$.

\smallskip

Conversely let $(st)' \in W(st) \cap D$. Then $t(st)'st(st)' = t(st)'$ and so $t(st)'$ is a weak inverse of $s$. As $t(st)' \mathcal{L} (st)'$ we have $W(s) \cap D \neq \emptyset$. Similarly we have $(st)'s \in W(t) \cap D \neq \emptyset$.
\end{proof}

\begin{corollary} \label{semilattice-cor}
Let $S$ be a conditionally completely regular semigroup and let $s, t \in S$. Then $s \rho s^2$ and $st \rho ts$.
\end{corollary}

\begin{corollary}
Let $S$ be a conditionally completely regular semigroup. Either $S$ is $E-$dense or the set $\{s \in S | W(s) = \emptyset\}$ is an ideal of $S$.
\end{corollary}

We can now prove the following theorem.

\begin{theorem}\label{semilattice-theorem}
Let $S$ be a conditionally completely regular semigroup. Then the relation $\rho$ is a congruence and $S/\rho$ is a semilattice.
\end{theorem}

\begin{proof}
Let $a,b,c,d \in S$ such that $a \rho b$ and $c \rho d$ and let $D$ be a $\cal D$-class of $S$. By Lemma \ref{rho-mult-lemma} $W(ac) \cap D \neq \emptyset$ if and only if $W(a) \cap D \neq \emptyset$ and $W(c) \cap D \neq \emptyset$. As $a \rho b$ and $c \rho d$ this latter condition is equivalent to $W(b) \cap D \neq \emptyset$ and $W(d) \cap D \neq \emptyset$ which is in turn equivalent to $W(bd) \cap D \neq \emptyset$ by Lemma \ref{rho-mult-lemma}. It follows that $ac \rho bd$ and so $\rho$ is a congruence. That $S/\rho$ is a semilattice follows from Corollary \ref{semilattice-cor}.
\end{proof}

We can now prove some results about the structure of $S$.

\begin{lemma}\label{rho-equals-D-lemma}
Let $S$ be a conditionally completely regular semigroup and let $s, t \in \Reg(S)$. Then $s \rho t$ if and only if $s \mathcal{D} t$.
\end{lemma}

\begin{proof}
From Lemma~\ref{rho-mult-lemma} it follows that all of Green's relations are contained in $\rho$. To see this suppose that $(s,t)\in{\cal J}$. Then there exists $u,v\in S^1$ such that $s = utv$. So for every ${\cal D-}$class $D$, if $W(s)\cap D\ne\emptyset$ then $W(utv)\cap D\ne\emptyset$. Hence by Lemma~\ref{rho-mult-lemma} $W(t)\cap D\ne\emptyset$. By a dual argument we then deduce that $W(s)\cap D\ne\emptyset$ if and only if $W(t)\cap D\ne\emptyset$ and so $(s,t)\in\rho$.

As $s$ is regular it has an inverse which lies in the same $\cal D$-class and so $W(s) \cap D_s \neq \emptyset$. Hence $W(t) \cap D_s \neq \emptyset$ and by Lemma \ref{D-class-lemma} there exists $t' \in W(t)$ such that $t' \mathcal{L} s$. By a similar argument there exists $s' \in W(s)$ such that $s' \mathcal{R} t$. Then
$$
s \mathcal{L} t' \mathcal{R} t't \mathcal{L} st \mathcal{R} ss' \mathcal{L} s' \mathcal{R} t
$$
and so $s \mathcal{D} t$ as required.
\end{proof}

It follows that for each $\rho$-class $S_\alpha$ either $S_\alpha$ has no regular elements or the regular elements in $S_\alpha$ are contained within a single $\cal D$-class and hence by Lemma \ref{D-class-lemma} form a completely simple subsemigroup of $S_\alpha$. In the latter case $S_\alpha$ is an $E-$dense semigroup as by definition of $\rho$ each element has a weak inverse lying in the regular $\cal D$-class. Since each $\cal J$-class is contained within a $\rho$-class, it also follows that the regular $\cal J$-classes of $S$ are exactly the regular $\cal D$-classes.

\begin{lemma}\label{e-dense-subsemigroup-lemma}
Let $S$ be a conditionally completely regular semigroup and let $x\in S$. Then $x\rho$ is an $E-$dense subsemigroup of $S$ if and only if $x\rho$ contains a regular element.
\end{lemma}
\begin{proof}
One way round is obvious. That $x\rho$ is a subsemigroup of $S$ follows from the fact that $\rho$ is a semilattice. Let $y\in x\rho$ be regular. Then there exists $y'\in W(y)\cap D_y$ and so for any $z\in x\rho$ there exists $z'\in W(z)\cap D_y$. Since $D_y\subseteq x\rho$ then $x\rho$ is $E-$dense.
\end{proof}

\begin{lemma}\label{regS-ideal-lemma}
Let $S$ be an $E-$dense semigroup such that $\Reg(S)$ is a completely simple semigroup. Then $\Reg(S)$ is an ideal of $S$. 
\end{lemma}

\begin{proof}
Let $s \in \Reg(S)$ and $t \in S$. Let $t' \in W(t)$ and let ${\cal H}$ be Green's ${\cal H}-$relation on $\Reg(S)$. As $\Reg(S)$ is completely simple every regular $\cal H$-class contains an inverse of $s$ so we may choose $s' \in V(s)$ such that $s' \mathcal{R} tt'$. Then $t's'stt's' = t's'ss' = t's'$ and $stt's'st = ss'st = st$. Hence $st$ is regular and so $\Reg(S)$ is a right ideal. A dual argument shows $\Reg(S)$ is a left ideal and hence an ideal.
\end{proof}

\begin{lemma}
Let $S$ be a conditionally completely regular semigroup, let $\alpha\in S/\rho$, let $S_\alpha={\rho^{\natural}}^{-1}(\alpha)$ and let $s \in S_\alpha$. Then for all $\beta\in S/\rho$, $S_\beta$ contains a weak inverse of $s$ if and only if $S_\beta$ contains a regular element and $\beta \leq \alpha$.
\end{lemma}

\begin{proof}
Suppose $\beta \leq \alpha$ and $S_\beta$ contains regular elements. Let $t \in S_\beta$. Then $st \in S_{\alpha\beta} = S_\beta$. As $S_\beta$ contains a regular element it is $E-$dense by Lemma~\ref{e-dense-subsemigroup-lemma}, and so there exists $(st)' \in W(st) \cap S_\beta$. Then $t(st)'st(st)' = t(st)'$ so $t(st)' \in W(s) \cap S_\beta$ as required. Conversely, let $s' \in W(s) \cap S_\beta$. Clearly $s'$ is regular, and $s' = s'ss' \in S_{\beta\alpha\beta} = S_{\alpha\beta}$ so $\alpha\beta = \beta$ and hence $\beta \leq \alpha$ as required.
\end{proof}

From the perspective of stratified extensions, we cannot say anything about these semigroups in general. For example, a free semigroup $S$ and a group $G$ both satisfy the property that every regular ${\cal H}$-class contains an idempotent, but $\Base(S) = \emptyset$ and $\Base(G) = G$. One condition that allows us to make more precise statements is to require that $S$ is a group-bound semigroup. Note that group-bound implies eventually regular, and when every regular ${\cal H}$-class contains an idempotent the two concepts are equivalent.

We will show that applying our results to a semigroup which is also group-bound gives the same decomposition as that in Theorem~\ref{shevrin-theorem-3}.

If $S$ is a group-bound semigroup and $e \in E(S)$ then let $H_e$ denote the largest subgroup of $S$ containing $e$. The set of elements $s$ such that $s^n \in H_e$ for some $n \in \mathbb{N}$ is denoted by $K_e$. This is well defined in the sense that if $s^n \in H_e$ we have $s^m \in H_e$ for all $m>n$ (\cite[Lemma 1]{shevrin-95}). It also follows that the sets $K_e$ partition $S$. In general $K_e$ is not a subsemigroup of $S$ (\cite[Proposition 7]{shevrin-95}) and in addition in a group bound semigroup ${\cal D} = {\cal J}$ (\cite[Lemma 4]{shevrin-95}). As is usual, $J_s$ will denote the $\cal J-$class of $s$.

The following result is important in what follows.

\begin{lemma} \label{J-lattice-lemma}
Let $S$ be an eventually regular conditionally completely regular semigroup. If $s \in K_e$ then $J_e$ is the greatest $\mathcal{J}$-class containing a weak inverse of $s$. Moreover, if $e \mathcal{J} f$  and $s \in K_e$ and $t \in K_f$ then $(s,t)\in\rho$.
\end{lemma}

\begin{proof}
Let $S$ be a semigroup satisfying the conditions stated. As $S$ is eventually regular and every regular element lies in a group then $S$ is group-bound. Let $s \in K_e$ for some idempotent $e$ so that there exists $n\in\n$ such that $s^n\in H_e$. Then
$$
(s^n(s^{n+1})^{-1})s(s^n(s^{n+1})^{-1}) = s^n(s^{n+1})^{-1}e = s^n(s^{n+1})^{-1}
$$
where $(s^{n+1})^{-1}$ is the inverse of $s^{n+1}$ in $H_e$. Therefore $s$ has a weak inverse in $H_e$ and hence in $J_e$.

Now let $s_1\in W(s)$ and notice that $s_1$ is regular and so lies in a group $H_f$, say. By Lemma \ref{D-class-lemma} every $\mathcal{H}$-class of $J_f$ contains a weak inverse of $s$. Let $s_2$ be a weak inverse of $s$ such that $s_2 \mathcal{L} s_1s$ and note that $s_2\in D_f=J_f$. Then as $s_2s_1s = s_2$ we have
$$
s_2s_1s^2s_2s_1 = s_2ss_2s_1 = s_2s_1,
$$
so $s_2s_1 \in W(s^2)$ and by Lemma \ref{D-class-lemma}, $s_2s_1 \in J_f$. Let us denote $w_2=s_2s_1$ so that we proceed inductively as follows. Suppose that for $n\ge 2$ we have $w_{n-1}\in W(s^{n-1})\cap J_f$. Let $s_n\in W(s)\cap L_{w_{n-1}s^{n-1}}$ and let $w_n=s_nw_{n-1}$, so that $w_ns^{n-1}=s_n$ and
$$
w_ns^nw_n = s_nsw_n = s_nss_nw_{n-1} = s_nw_{n-1}=w_n.
$$
Hence $w_n\in W(s^n)\cap J_f$.

\smallskip

We see then that there is a weak inverse of $s^n$ in $J_f$ for any $n \in \n$. In particular, since $s\in K_e$, we can choose $n$ large enough such that $s^n \in H_e\subseteq J_e$. Let $s^\ast$ be the associated weak inverse of $s^n$ in $J_f$. Then by Lemma \ref{E-dense-J-lemma} we have $J_f = J_{s^\ast} \leq J_{s^n} = J_e$. Consequently if $s \in K_e$ then $J_e$ is the greatest $\mathcal{J}$-class containing a weak inverse of $s$.

\smallskip

Now let $s \in K_e$ and $t \in K_f$ as in the statement of the lemma. We can assume that $s$ and $t$ are regular. To see this, let $n\in\n$ be the minimum value such that $s^n\in H_e$ and note that if $(s^n)'$ is a weak inverse of $s^n$ then $s^{n-1}(s^n)'$ is a weak inverse of $s$ with $s^{n-1}(s^n)' \mathcal{L} (s^n)'$. This, along with the previous argument, shows that $s$ has a weak inverse in a ${\cal J}-$class $J$ if and only if the regular element $s^n$ has a weak inverse in $J$.

Let $J$ be a ${\cal J}-$class containing a weak inverse $s'$ of $s$. If $t \mathcal{L} s$ then $ts' \mathcal{L} ss'$ and so $ts' \in J$. Then, since $J$ is regular, there exists $r \in J$ such that $ts'r \in J$ is an idempotent, and so $s'rts'r \in J$ is a weak inverse of $t$. By a similar argument if $t \mathcal{R} s$ there is a weak inverse of $t$ in $J$ and so if $s \mathcal{J} t$ there is a weak inverse of $t$ in $J$. A dual argument then gives the opposite direction, and the result follows from the definition of $\rho$.
\end{proof}

Note that each $\mathcal{H}$-class of $S$ contains at most one weak inverse of $s$: if $s', s^\ast  \in W(s)$ with $s' \mathcal{H} s^\ast $ then $s's \mathcal{R} s' \mathcal{R} s^\ast  \mathcal{R} s^\ast s$. As $\mathcal{L}$ is a right congruence we also have $s's \mathcal{L} s^\ast s$. Since $s's$ and $s^\ast s$ are idempotents it follows that $s's = s^\ast s$ and by a similar argument $ss' = ss^\ast $. Then $s' = s'ss' = s^\ast ss' = s^\ast ss^\ast  = s^\ast $.

\begin{theorem}\label{KJe-semilattice-prop}
Let $S$ be a conditionally completely regular semigroup. If $S$ is group-bound then $S$ is a semilattice of Archimedean semigroups of the form $K_{J_e} = \bigcup_{f \in E(J_e)} K_f$ for $e\in E(S)$.
\end{theorem}

\begin{proof}
Let $e\in E(S)$ and define $K_{J_e} = \bigcup_{f \in E(J_e)} K_f$.
Let $s,t \in K_{J_e}$ and notice that $s\in K_f, t\in K_g$ for some $f,g\in J_e$, so that by Lemma~\ref{J-lattice-lemma}, $(s,t)\in\rho$. 
Conversely, if $(s,t)\in\rho$ then there exists $e,f\in E(S)$ such that $s\in K_e\subseteq K_{J_e}, t\in K_f\subseteq K_{J_f}$. By Lemma~\ref{J-lattice-lemma}, $J_e$ is the greatest ${\cal J}-$class containing a weak inverse of $s$ and $J_f$ is the greatest ${\cal J}-$class containing a weak inverse of $t$. Since $(s,t)\in\rho$ it easily follows that $J_e=J_f$ and so $s,t \in K_{J_e}=K_{J_f}$. Hence the sets $K_{J_e}$ are the $\rho-$classes and so partition $S$ and since $S/\rho$ is a semilattice then the result follows.

\smallskip

For each $e,f \in E(S)$ it follows that there exists $g\in E(S)$ such that $K_{J_e}K_{J_f}\subseteq K_{J_g}$. Since $e\in K_{J_e}$ and $f\in K_{J_f}$ then $ef\in K_{J_g}$. In addition there exists a uniquely determined $h\in E(S)$ such that $ef\in K_h\subseteq K_{J_h}$ and so $K_{J_g}=K_{J_h}$.

To see that $K_{J_e}$ is an Archimedean semigroup, let $a,b \in K_{J_e}$. Then there exist $m,n \in \n$ such that $a^m, b^n \in J_e$ and so $a^m \in K_{J_e} b^n K_{J_e} \subseteq K_{J_e} b K_{J_e}$ as required.
\end{proof}

Note that a decomposition into a semilattice of Archimedean semigroups is necessarily unique: Let $S=[Y;S_\alpha] = [Y';S_a]$ be two Archimedean semilattice decompositions of the semigroup $S$. If $s,t\in S$ lie in the same subsemigroup $S_\alpha$ where $\alpha\in Y$ and $s \in S_a, t \in S_b$ where $a,b\in Y'$, then there exist $n \in \n$ and $u,v \in S$ such that $s^n = utv$ and $a \leq b$. Similarly $b \leq a$ and so $s,t \in S_a$ and the two semilattices, $Y$ and $Y'$, are isomorphic. We have hence recovered the same decomposition as Shevrin (Theorem~\ref{shevrin-theorem-3}) in this case.

\smallskip

It is clear from the above structure that these semigroups are group-bound and since it is straightforward to check that $\Reg(K_{J_e}) = J_e$, then the regular elements form a completely simple subsemigroup.

\smallskip

The converse of Theorem~\ref{KJe-semilattice-prop} does not hold in general as an Archimedean semigroup need not contain regular elements and hence a semilattice of Archimedean semigroups may not be group-bound. It is enough, however, to require that each Archimedean semigroup contains a regular element.

\begin{corollary} \label{KJe-semilattie-cor}
Let $S$ be a conditionally completely regular semigroup. Then $S$ is group-bound if and only if $S=[Y;S_\alpha]$ is a semilattice of Archimedean semigroups $S_\alpha$ with $\Reg(S_\alpha) \neq \emptyset$.
\end{corollary}

\begin{proof}
Clearly if $S$ is group-bound then every subsemigroup contains a regular element. Conversely, let $s \in S$. Then $s \in S_\alpha$ for some $\alpha$ and let $t \in \Reg(S_\alpha)$. Since $S_\alpha$ is an Archimedean semigroup, there exists $n \in \n$ such that $s^n \in S_\alpha t S_\alpha$ and hence $s^n \in \Reg(S_\alpha) \subseteq \Reg(S)$ by Lemmas~\ref{e-dense-subsemigroup-lemma} and \ref{regS-ideal-lemma}.
\end{proof}

\begin{proposition}\label{group-bound-proposition}
Let $S$ be a semigroup. Any two of the following implies the third.
\begin{enumerate}
\item $S$ is group-bound;
\item $S$ is conditionally completely regular;
\item $S$ is a semilattice of Archimedean semigroups $S_\alpha$ with $\Reg(S_\alpha) \neq \emptyset$.
\end{enumerate}
\end{proposition}

\begin{proof}
By Corollary~\ref{KJe-semilattie-cor} we have (1) and (2) imply (3) and (2) and (3) imply (1). The remaining implication follows from Theorem~\ref{shevrin-theorem-3}.
\end{proof}

\medskip

We now turn our attention to describing the subsemigroups $K_{J_e}$ at each vertex of the semilattice. Since each semigroup contains regular elements, they are all stratified extensions with a base consisting of at least the regular elements. From \cite[Proposition 3]{shevrin-95} each $K_{J_e}$ is an ideal extension of the completely simple semigroup $J_e$ by a nilsemigroup. If this nilsemigroup is stratified then $K_{J_e}$ is a nil-stratified extension with base $J_e$. However not every nilsemigroup is stratified as demonstrated by the following example.

\medskip

Let $S={\cal P}(\n)$ be the set of subsets of $\n$ and define a multiplication on $S$ by
$$
A\circ B = \begin{cases}A\cup B&\text{if }A\text{ and }B\text{ are non-empty and }A\cap B=\emptyset\\\emptyset&\text{otherwise.}\end{cases}
$$
Then it is clear that each element is nilpotent of index 2 and so $S$ is a nilsemigroup, but all infinite elements are in the base. Hence $S$ is not stratified.

\begin{lemma}
Let $S$ be an eventually regular semigroup such that $\Reg(S)$ is completely simple and suppose $S$ is a finitely stratified extension. Then $\Base(S) \setminus \Reg(S)$ is either empty or infinite.
\end{lemma}

\begin{proof}
Suppose $s_0\in\Base(S) \setminus \Reg(S) \neq \emptyset$. Since $S$ is a finitely stratified extension, $\Base(S)$ is a globally idempotent subsemigroup so $s_0=s_{1}t_{1}$ for some $s_{1},t_{1} \in \Base(S)$. If $s_{1}$ is regular then as $\Reg(S)$ is an ideal, $s_0$ is regular giving a contradiction. Further, if $s_{1}=s_0$ then $s_0=s_0t_{1}=s_0{t_{1}}^n$ for any $n\in\n$. We can choose $n$ such that ${t_1}^n$ is regular, so $s_0$ is again regular giving a contradiction. Hence $s_{1}$ is an element of $\Base(S) \setminus \Reg(S)$ not equal to $s_0$. By a similar argument, $s_{1} = s_{2}t_{2}$ where $s_{2}\in \Base(S) \setminus \Reg(S)$ and $s_2$ is not equal to $s_0$ nor $s_{1}$. Proceeding inductively we deduce that the set $\{s_0, s_1, s_2, \dots \}$ is an infinite subset of $\Base(S) \setminus \Reg(S)$.
\end{proof}

It follows that any finite semigroup in which  every regular $\mathcal{H}$-class contains an idempotent
is a semilattice of finitely stratified extensions with completely simple bases.

\begin{theorem}
A semigroup $S$ is a finite conditionally completely regular semigroup if and only if $S=[Y;S_\alpha]$ is a finite semilattice of finite semigroups $S_\alpha$ where each $S_\alpha$ is a finitely stratified extension of a completely simple semigroup.
\end{theorem}

\begin{proof}
To see that the converse is true, let $s\in S$ be a regular element, so that there exists $\alpha$ such that $s\in S_\alpha$. Let $s'$ be an inverse of $s$ (within $S$) with $s'\in S_\beta$ for some $\beta$. Then $s = ss's\in S_\alpha S_\beta S_\alpha \subseteq S_{\alpha\beta}\cap S_\alpha$, and so $S_\alpha = S_{\alpha\beta}$. Similarly $s'=s'ss'\in S_{\alpha\beta}\cap S_\beta$ and so $S_\alpha = S_\beta$. It follows that $s$ is regular within $S_\alpha$ and so $s\in \Base(S_\alpha)$ and is therefore ${\cal H}-$related to an idempotent as required.
\end{proof}

\section{Strict extensions of Clifford Semigroups}

In this section we continue our exploration of examples by looking at Clifford semigroups. In particular, we show that every strict extension of a Clifford semigroup can be decomposed as a semilattice of stratified extensions of groups.

\smallskip

We make use of the notation of Clifford and Preston \cite[Section 4.4]{clifford-preston-61}, and in particular that relating to ideal extensions determined by partial homomorphisms. A Clifford semigroup is a completely regular inverse semigroup. It is well known that a Clifford semigroup $S$ decomposes as a semilattice of groups $S={\cal S}[Y;G_\alpha]$. We begin by showing that a strict extension $\Sigma$ of a Clifford semigroup $S$ has a semilattice structure isomorphic to that of the Clifford semigroup itself.

\begin{lemma}
Let $S={\cal S}[Y;G_\alpha]$ be a Clifford semigroup. An ideal extension of $S$ is strict if and only if it is determined by a partial homomorphism.
\end{lemma}

\begin{proof}
Let $a,b \in S$ be such that $ax=bx$ and $xa=xb$ for all $x \in S$. As $S$ is a Clifford semigroup $a \in G_\alpha$ and $b \in G_\beta$ for some $\alpha, \beta \in Y$. Let $e,f$ be the identities of $G_\alpha, G_\beta$ respectively. Then $a = ea = eb$ and so $\alpha \leq \beta$. Similarly, $b = fb = fa$ so $\beta \leq \alpha$ and so $\alpha = \beta$ and $e=f$. Then $a=ea=eb=b$ and hence $S$ is weakly reductive. The result then follows from Theorem~\ref{grillet-petrich-theorem}.
\end{proof}

\begin{lemma}\label{Clifford-extension-lemma}
Let $\Sigma$ be a strict extension of a Clifford semigroup $S={\cal S}[Y;G_\alpha]$ by a semigroup $T$ defined by a partial homomorphism $A \mapsto \overline{A}$ and let $\Sigma_\alpha = G_\alpha \cup \{A \in T\setminus\{0\} | \overline{A} \in G_\alpha\}$ for each $\alpha \in Y$. Define a relation $\sim$ on $\Sigma$ by $s \sim t$ if and only if $s,t \in \Sigma_\alpha$ for some $\alpha \in Y$. Then $\sim$ is a congruence and $\Sigma/\!\!\sim$ is a semilattice isomorphic to $Y$.
\end{lemma}

\begin{proof}
Clearly $\sim$ is an equivalence relation. To prove $\sim$ is a congruence and that $\Sigma/\!\!\sim\; \cong Y$ we show that $\sim$ is the kernel of the homomorphism $\theta: \Sigma \rightarrow Y$ where if $s \in \Sigma_\alpha$ then $\theta(s) = \alpha$. Note that if $A \in T\setminus\{0\}$ then $\theta(A) = \theta(\overline{A})$. We have four cases to consider:
\begin{enumerate}
\item If $s,t \in S$ then $\theta(s)\theta(t) = \theta(st)$ follows from the semilattice structure of $S$.
\item If $s \in S$ and $A \in T\setminus\{0\}$ then $\theta(s)\theta(A) = \theta(s)\theta(\overline{A}) = \theta(s\overline{A}) = \theta(sA)$, where the last two equalities follow from the first case and multiplication in a strict extension respectively. The case for $\theta(A)\theta(s)$ follows similarly.
\item If $A,B \in T\setminus\{0\}$ then $\theta(A)\theta(B) = \theta(\overline{A})\theta(\overline{B}) = \theta(\overline{A}\;\overline{B})$ by the first case. Then if $AB = 0$ in $T$ we have $\theta(AB) = \theta(\overline{A}\;\overline{B})$ and if $AB \neq 0$ in $T$ we have $\theta(AB) = \theta(\overline{AB}) = \theta(\overline{A}\;\overline{B})$. In either case $\theta(A)\theta(B) = \theta(AB)$.
\end{enumerate}
Hence $\theta$ is a homomorphism as required and $\sim$ is clearly its kernel.
\end{proof}

\begin{theorem}\label{clifford-extension-semilattice-theorem}
Every strict extension $\Sigma$ of a Clifford semigroup $S$ by a semigroup $T$ is a semilattice of extensions of groups. Conversely, if $\Sigma$ is a semilattice of extensions $\Sigma_\alpha$ of groups $G_\alpha$ and $S = \bigcup_{\alpha\in Y}G_\alpha$ is an ideal of $\Sigma$ then $\Sigma$ is a strict extension of the Clifford semigroup $S$.
\end{theorem}

\begin{proof}
By Lemma~\ref{Clifford-extension-lemma}, $\Sigma$ is a semilattice of semigroups $\Sigma_\alpha$ defined via a partial homomorphism $A\mapsto {\overline{A}}$ from $T\setminus\{0\}\to S$. The restriction of this map to $\Sigma_\alpha \setminus G_\alpha$ gives a partial homomorphism defining the ideal extension $\Sigma_\alpha$ of the group $G_\alpha$.

\smallskip

Conversely, let $\Sigma$ be a semilattice of semigroups $\Sigma_\alpha$ where each $\Sigma_\alpha$ is an ideal extension of a group $G_\alpha$ by a stratified semigroup $T_\alpha$ and $S = \bigcup_{\alpha \in Y} G_\alpha$ is an ideal of $\Sigma$. It follows that $S$ is a Clifford semigroup and $\Sigma$ is an ideal extension of $S$ by $T = \Sigma/S$, where $T$ can equivalently be viewed as $\{0\} \cup \bigcup_{\alpha \in Y} T_\alpha\setminus\{0\}$. As $G_\alpha$ has identity $e_\alpha$ the extension $\Sigma_\alpha$ is determined by the partial homomorphism $A \mapsto A e_\alpha \; (= e_\alpha A)$ (Proposition~\ref{grillet-petrich-proposition} and Theorem~\ref{grillet-petrich-theorem}). The union of these maps is then a map $\varphi: T\setminus\{0\} \rightarrow S$ such that $\varphi(A) = Ae_\alpha$ for each $A \in T_\alpha\setminus\{0\}$. We will show that $\varphi$ is a partial homomorphism and that it defines the ideal extension $\Sigma$. For clarity, the multiplication determined by $\varphi$ will be denoted by $\circ$, multiplication within $T$ by $\ast $, and the original multiplication of the semilattice $\Sigma$ by juxtaposition.

Let $A, B \in T\setminus\{0\}$ such that $A\ast B \neq 0$ and assume $A \in T_\alpha$, $B \in T_\beta$ so that $A\ast B \in T_{\alpha\beta}$. Then $\varphi(A)\varphi(B) = Ae_\alpha(Be_\beta) = A(Be_\beta) e_\alpha = ABe_{\alpha\beta} = \varphi(AB)$ as required.

This partial homomorphism determines an ideal extension of $S$ consisting of the same set $\Sigma$ under the multiplication $\circ$ defined by
\begin{enumerate}
\item $s \circ t = st$
\item $A \circ B = \begin{cases}
AB & \text{if }A\ast B \neq 0\\
\varphi(A)\;\varphi(B) & \text{otherwise}
\end{cases}$
\item $A \circ s = \varphi(A)s$
\item $s \circ A = s\varphi(A)$
\end{enumerate}
where $A,B \in T\setminus\{0\}$ and $s,t \in S$. We show that in all cases, this multiplication is equivalent to the original multiplication on $\Sigma$. The first condition and the first part of the second condition do not require proof. For the second part of the second condition, let $A \in T_\alpha\setminus\{0\}$ and $B \in T_\beta\setminus\{0\}$ with $A\ast B = 0$ so $AB \in G_{\alpha\beta}$. Then $$A \circ B = \varphi(A)\varphi(B) = Ae_\alpha (Be_\beta) = A(Be_\beta) e_\alpha = ABe_{\alpha\beta} = AB$$ as required.
For the third condition, let $A \in T_\alpha\setminus\{0\}$ and $s \in G_\beta$ with $As \in G_{\alpha\beta}$. Then $$A \circ s = \varphi(A)s = Ae_\alpha (s e_\beta) = A(se_\beta) e_\alpha = As e_{\alpha\beta} = As$$ as required. The fourth condition follows a dual argument. Hence $\varphi$ determines the extension $\Sigma$ and so it is a strict extension of $S$.
\end{proof}

\begin{corollary}\label{clifford-extension-semilattice-corollary}
Let $\Sigma$ be a strict stratified extension of a Clifford semigroup $S$. Then $\Sigma$ is a semilattice of stratified extensions of groups.
\end{corollary}

\begin{proof}
Let $\Sigma$ be a strict extension of a Clifford semigroup $S$ by a stratified semigroup $T$. By Theorem~\ref{clifford-extension-semilattice-theorem}, $\Sigma$ is a semilattice of semigroups $\Sigma_\alpha$, each of which is an ideal extension of a group $G_\alpha$ by a subsemigroup of $T$ containing zero. It can be easily verified that such a subsemigroup is also stratified, and hence $\Sigma$ is a semilattice of stratified extensions of groups.
\end{proof}

The converse of Corollary~\ref{clifford-extension-semilattice-corollary} does not hold in general as each $T_\alpha$ being a stratified semigroup does not guarantee that $T$ is itself a stratified semigroup. For example, let $Y = \{a,b\}$ with $a \leq b$. For each $\alpha \in Y$ let $G_\alpha$ be a group, $T_\alpha$ a free semigroup with adjoined zero, and $\Sigma_\alpha$ an ideal extension of $G_\alpha$ by $T_\alpha$. For $s \in T_a$ and $t \in T_b$ let $st=ts = s$. Along with the fact that $S = G_a \cup G_b$ is an ideal of $\Sigma$, this defines a multiplication on the semilattice $\Sigma = \Sigma_a \cup \Sigma_b$. Each $T_\alpha$ is a stratified semigroup so each $\Sigma_\alpha$ is a stratified extension of a group, however $T = \Sigma / S$ is not stratified, as $\bigcap_{i \geq 1} T^i \cong T_a$. A sufficient, but clearly not necessary, condition under which $T$ will always be stratified is if $T$ is finite.

\bigskip

\def\dom{\hbox{dom}}
As an example of the above construction, consider the following. Let $n \in \n$ and let $N = \{1,\dots,n\}$. Let $S=G_1^0\times\ldots\times G_n^0$ be a direct product of $0-$groups $G^0_i$, $i\in N$. For $s = (a_1,\ldots,a_n)\in S$ define $\dom(s) = \{i\in N|a_i\ne0\}$.

\smallskip

Let $m\in\n$ and define a relation $\rho_m$ on $(\n, +)$ by
$$
\rho_m = 1_{\n}\cup \{(x,y)\in\n\times\n|x,y\ge m\}.
$$
Then it is easy to check that $S$ is a Clifford semigroup (and hence a strong semilattice of groups), $\rho_m$ is a congruence on $\n$ and $\n/\rho_m$ is a finite monogenic semigroup with trivial kernel. For simplicity, we shall identify $\n/\rho_m$ with $\{1,\ldots,m\}$, in the obvious way. Let $T'$ be the semigroup of all partial maps from $N$ to $\n/\rho_m$ with binary operation $\ast$ given by $(f\ast g)(x) = f(x)+g(x)$ when both are defined and undefined otherwise. Let $I \subseteq T'$ be the set of maps whose image is $\{m\}$. It can be readily seen that $I$ is an ideal of $T'$ and $T = T'/I$ is a nilsemigroup.

\smallskip

For each $i \in N$ pick an element $g_i \in G_i$ and let $\alpha_i: T\setminus\{0\} \rightarrow G^0_i$ be the partial homomorphism given by $$\alpha_i(f) = 
\begin{cases}
g_i^{f(i)} & f(i) \text{ is defined} \\
0 & \text{otherwise.}
\end{cases}$$ Then $\alpha: T\setminus\{0\} \rightarrow S$ given by $\alpha(f) = (\alpha_1(f),\dots,\alpha_n(f))$ is a partial homomorphism defining an ideal extension $\Sigma$ of $S$ by $T$.

\smallskip
Notice that $s{\cal J}t$ if and only if $\dom(s)=\dom(t)$.
It follows that the semilattice structure of $S$ is defined in terms of the power set of $N$ (i.e. $\dom(st) = \dom(s)\cap\dom(t)$). Let $S_M$ be the $\cal J$-class of $S$ with $\dom(s)=M$ for $s\in S_M$. Then $T_M = \alpha^{-1}(S_M)$ is the set of maps in $T\setminus\{0\}$ whose domain is exactly $M$. The set $T_M^0=T_M\cup\{0\}$ is a subsemigroup of $T$ and is a nilsemigroup. The restriction of $\alpha$ to $T_M$ then gives a partial homomorphism from $T_M^0$ to $S_M$ which defines an ideal extension $\Sigma_M$ of the group $S_M$ by $T_M^0$. It can then be shown that $\Sigma$ is a semilattice of these semigroups $\Sigma_M$.

\bigskip

The authors thank the anonymous referee for pointing out the example after Proposition~\ref{group-bound-proposition} and for a number of useful comments which have enhanced the overall exposition of the paper.

\end{document}